\documentclass[12pt,a4paper]{article}
\usepackage{amsfonts}
\usepackage{amssymb}
\usepackage{amsthm}
\usepackage[english]{babel}
\usepackage{hhline}
\usepackage[ansinew]{inputenc}
\usepackage{amsmath}
\usepackage[all,2cell,ps]{xy}
\usepackage{qsymbols}
\usepackage{color}
\usepackage{epsfig}
\usepackage{color}
\usepackage{graphics}
\usepackage{graphicx}%
\usepackage{enumerate}
\usepackage{mathrsfs}
\usepackage{amssymb}
\usepackage[affil-it]{authblk}
\usepackage{etoolbox}

\theoremstyle{plain}
\newtheorem{thm}{Theorem}
\newtheorem{cor}[thm]{Corollary}
\newtheorem{lem}[thm]{Lemma}
\newtheorem{defn}{Definition}
\newtheorem{rem}{Remark}
\newtheorem{prop}{Proposition}

\newcommand{\ra}{\rightarrow}

\newcommand{\we}{\wedge}
\newcommand{\X}{\mathbf{X}}
\newcommand{\D}{\mathbf{D}}
\newcommand{\KA}{\mbox{{\Large $\kappa$}}}
\newcommand{\Ra}{\Rightarrow}
\newcommand{\f}{(\ )^{\IS}}
\newcommand{\fs}{(\ )^{\GHey}}
\newcommand{\fsz}{(\ )^{\Hey}}
\newcommand{\fHil}{(\ )^{\dag}}

\newcommand{\clx}{\overline{\{x\}}}
\newcommand{\cly}{\overline{\{y\}}}
\newcommand{\sat}{\mathrm{sat}}
\newcommand{\pd}{\preceq_{d}}

\newcommand{\HIS}{H^{\mathsf{IS}}}
\newcommand{\GIS}{G^{\mathsf{IS}}}
\newcommand{\fIS}{f^{\mathsf{IS}}}
\newcommand{\gIS}{g^{\mathsf{IS}}}

\newcommand{\Fil}{\mathsf{Fil}}
\newcommand{\Id}{\mathsf{Id}}
\newcommand{\Xs}{X^{\vee}}
\newcommand{\Rfs}{R_{f}^{\vee}}

\newcommand{\K}{\mathrm{K}}

\newcommand{\U}{\mathrm{U}}
\newcommand{\UIS}{\mathrm{U}}
\newcommand{\UGHey}{\mathrm{U}}
\newcommand{\UHey}{\mathrm{U}}
\newcommand{\Sg}{\mathrm{Sg}}
\newcommand{\IS}{\mathsf{IS}}
\newcommand{\Hey}{\mathsf{Hey}}
\newcommand{\GHey}{\mathsf{gH}}
\newcommand{\Hil}{\mathsf{Hil}}
\newcommand{\Hils}{\mathsf{Hil}^{\vee}}
\newcommand{\Hilsz}{\mathsf{Hil_{0}}^{\vee}}
\newcommand{\HS}{\mathsf{HS}}

\newcommand{\IHil}{\mathrm{I_{Hil}}}
\newcommand{\FC}{\mathrm{S}}

\newcommand{\Ser}{\mathrm{S}}

\newcommand{\EL}{\mathscr{L}}
\newcommand{\ELK}{\mathscr{L}_K}

\makeatletter
\patchcmd{\maketitle}
 {\def\@makefnmark}
 {\def\@makefnmark{}\def\useless@macro}
 {}{}
\makeatother

\begin{document}

\title{Variations of the free implicative semilattice extension of a Hilbert algebra}

\author{Jos\'e L. Castiglioni\thanks{Departamento de Matem\'atica,
Facultad de Ciencias Exactas, Universidad Nacional de La Plata and
Conicet-Argentina, jlc@mate.unlp.edu.ar.} \ and \ Hern\'an J. San Mart\'{\i}n
\thanks{Departamento de Matem\'atica, Facultad de Ciencias Exactas,
Universidad Nacional de La Plata and Conicet-Argentina, hsanmartin@mate.unlp.edu.ar.}}
\maketitle

\begin{abstract}
In [{\it On the free implicative semilattice extension of a
Hilbert algebra}. Mathematical Logic Quarterly 58, 3 (2012),
188--207], Celani and Jansana give an explicit description of the
free implicative semilattice extension of a Hilbert algebra. In
this paper we give an alternative path conducing to this
construction. Furthermore, following our procedure, we show that
an adjunction can be obtained between the algebraic categories of
Hilbert algebras with supremum and that of generalized Heyting
algebras. Finally, in last section we describe a functor from the
algebraic category of Hilbert algebras to that of generalized
Heyting algebras, of possible independent interest.
\end{abstract}

\section{Introduction}

In what follows we assume the reader is familiar with the theory
of Heyting algebras \cite{BD}, which are the algebraic counterpart
of Intuitionistic Propositional Logic. Hilbert algebras were
introduced in the early 50's by Henkin for some investigations of
implication in intuitionistic and other non-classical logics
(\cite{R}, pp. 16). In the 60's, they were studied especially by
Horn \cite{H} and Diego \cite{D}.

Let $K$ be any variety in the language $\ELK$ and $\EL$
a sublanguage of $\ELK$. We may consider the variety generated by the
$\EL$-reducts of the elements of $K$; let us write $M$ for this variety.
Let us also write $K$ and $M$, for the algebraic categories whose class of
objects are the members of the varieties $K$ and $M$, respectively. The correspondence
assigning each member of $K$ to its $\EL$-reduct induces a functor $U: K \to M$,
which is usually referred to as the forgetful functor from $K$ to $M$,
since part of the structure of each member in $K$ is forgotten. It can be seen that
this functor has a left adjoint $F: M \to K$ (see, for instance, \cite[Theorem 3.8]{Mo17}).
For any $a \in M$, the element $F(a) \in K$ is usually referred to as the free $K$-extension of $a$.
However, the usual general arguments guaranteing its existence do not provide an easy description of it.

In \cite{CJ2} Celani and Jansana gave a concrete description of the
free implicative semilattice extension of a Hilbert algebra, from where
an explicit description for the left adjoint of the forgetful functor from the category of
implicative semilattices to the category of Hilbert algebras follows.

The main goal of this paper is to arrive at the explicit description of the ajunction
presented in \cite{CJ2} following an alternative path.
We also apply these ideas in order to provide similar
constructions for the category of generalized Heyting algebras.

The paper is organized as follows. In Section \ref{BR} we give
some basic results about Hilbert algebras. In particular, we
recall the categorical equivalence for Hilbert algebras developed
by Cabrer, Celani and Montangie in \cite{CCM} (see also
\cite{CMs}). In Section \ref{s1} we use the equivalence for the
category of Hilbert algebras in order to build up a functor from
the category of Hilbert algebras to the category of implicative
semilattices. We also present an explicit description for the
left adjoint to the forgetful functor from the category of
implicative semilattices to the category of Hilbert algebras.
Finally we establish the connections between our results and those of
\cite{CJ2}. In Section \ref{s2} we give an explicit description
for the left adjoint to the forgetful functor from the category of
generalized Heyting algebras (Heyting algebras) to the category of
Hilbert algebras with supremum (Hilbert algebras with supremum and
a minimum). Finally, in Section \ref{fr} we build up a functor
from the category of Hilbert algebras to the category of
generalized Heyting algebras and we comment some open problems.

\section{Basic results} \label{BR}

Recall that a \emph{Hilbert algebra} (see \cite{D}) is an algebra
$(H,\ra,1)$ of type $(2,0)$ which satisfies the following
conditions for every $a,b,c\in H$:
\begin{enumerate}[\normalfont a)]
\item $a\ra (b\ra a) = 1$, \item $(a\ra (b\ra c)) \ra ((a\ra b)\ra
(a\ra c)) = 1$, \item if $a\ra b = b\ra a = 1$ then $a = b$.
\end{enumerate}

In \cite{D} Diego proves that the class of Hilbert algebras is a
variety. Moreover, this is the variety generated by the
$\{1,\ra\}$-reduct of Heyting algebras. In every Hilbert algebra
$H$ we have a partial order given by $a\leq b$ if and only if
$a\ra b = 1$, which is called \emph{natural order}. Relative to
the natural order on $H$, $1$ is the greatest element.

We write $\Hil$ for the variety of Hilbert algebras. Observe that
$\Hil$, when equipped with homomorphisms, has the structure of a
category.

\begin{lem}
Let $H\in \Hil$ and $a,b,c\in H$. Then the following conditions
are satisfied:
\begin{enumerate}[\normalfont a)]
\item $a\ra a = 1$, \item $1\ra a = a$, \item $a\ra (b\ra c) =
b\ra (a \ra c)$, \item $a\ra (b\ra c) = (a\ra b)\ra (a\ra c)$,
\item if $a\leq b$ then $c\ra a \leq c\ra b$ and $b\ra c \leq a\ra
c$.
\end{enumerate}
\end{lem}

Some additional elemental properties of Hilbert algebras can be
found in \cite{B,D}.

For the general development of Hilbert algebras, the notion of
implicative filter plays an important role. Let $H$ be a Hilbert
algebra. A subset $F\subseteq H$ is said to be a \emph{implicative
filter} if the following two conditions are satisfied: 1) $1\in
F$; 2) if $a\in F$ and $a\ra b \in F$ then $b\in F$. If in
addition $F\neq H$, then we say that the implicative filter $F$ is
\emph{proper}. Let $f:H\ra G$ be a function between Hilbert
algebras. In \cite[Theorem 3.2]{Cel} it was proved that the
following two conditions are equivalent: 1) $f(1) = 1$ and $f(a\ra
b) \leq f(a)\ra f(b)$ for every $a,b \in H$; 2) $f^{-1}(F)$ is an
implicative filter of $H$ whenever $F$ is an implicative filter of
$G$.

\begin{defn}
Let $H\in \Hil$ and $F$ an implicative filter. We say that $F$ is
irreducible if $F$ is proper and for any implicative filters $F_1,
F_2$ such that $F = F_1 \cap F_2$ we have that $F = F_1$ or $F =
F_2$. We write $X(H)$ for the set of irreducible implicative
filters of $H$.
\end{defn}

Let us consider a poset $\left\langle X,\leq\right\rangle$. A
subset $U\subseteq X$ is said to be an \emph{upset} if for all $x,
y \in X$ such that $x\in U$ and $x\leq y$ we have $y\in U$. The
notion of \emph{downset} is dually defined.

\begin{rem}
Every implicative filter of a Hilbert algebra is an upset.
\end{rem}

Let $H\in \Hil$ and $I \subseteq H$ with $I\neq \emptyset$. We say
that $I$ is an \emph{order-ideal} if $I$ is a downset and for
every $a,b\in I$ there is $c\in I$ such that $a\leq c$ and $b\leq
c$. Let $\Id(H)$ be the set of order-ideals of $H$ and $\Fil(H)$
the set of implicative filters of $H$. The following lemma is
\cite[Theorem 2.6]{Cel} (it is also an immediate consequence of
the work of polarities and optimality as carried on for instance
in \cite[Proposition 6.7]{GJP}).

\begin{lem}\label {tfp}
Let $H\in \Hil$. Let $F\in \Fil(H)$ and $I \in \Id(H)$ such that
$F\cap I = \emptyset$. Then there exists $P\in X(H)$ such that
$F\subseteq P$ and $P\cap I = \emptyset$.
\end{lem}

Recall that if $H$ is a Hilbert algebra and $X \subseteq H$, we
define the implicative filter generated by $X$ as the least filter
of $H$ that contains the set $X$, which will be denoted by $F(X)$.
There is an explicit description for $F(X)$ (see \cite[Lemma
2.3]{Bu}):
\[
F(X) =\{x \in H: a_1\ra (a_2\ra \cdots (a_n \ra x) \ldots)=1
\;\text{for some}\; a_{1}, \ldots,a_n\in X\}.
\]

The following known results are consequence of Lemma \ref{tfp}.

\begin{cor}\label{tfpc1}
Let $H\in \Hil$, $F \in \Fil(H)$ and $a \notin F$. Then there
exists $P\in X(H)$ such that $F\subseteq P$ and $a\notin P$.
\end{cor}

\begin{cor}\label{tfpc2}
Let $H\in \Hil$ and $a,b\in H$ such that $a\nleq b$. Then there
exists $P\in X(H)$ such that $a\in P$ and $b\notin P$.
\end{cor}

\begin{cor} \label{tfpc3}
Let $H\in \Hil$, $F\in \Fil(H)$ and $a,b\in H$. Then $a\ra b
\notin F$ if and only if there exists $P\in X(H)$ such that
$F\subseteq P$, $a\in P$ and $b\notin P$.
\end{cor}

If $f:H\ra G$ is a function between Hilbert algebras, we define
the relation $R_f \subseteq X(G)\times X(H)$ by
\[
(P,Q)\in R_f\;\text{if and only if}\; f^{-1}(P) \subseteq Q.
\]

Normally duals of homomorphisms are functions (e.g. in Priestley
and Stone dualities). However, these functions can be seen as
binary relations and, accordingly, in not so well-behaved
dualities the dual of a homomorphism tend to be a binary relation.
The definition of $R_f$ should be understood in this spirit.

The following lemma was proved in \cite[Theorem 3.3]{Cel}.

\begin{lem} \label{Th3.6}
Let $H$ and $G$ be Hilbert algebras and $f:H\ra G$ a function such
that $f(1) = 1$ and $f(a\ra b)\leq f(a)\ra f(b)$ for every $a,b\in
H$. Then the following statements are equivalent:
\begin{enumerate}[\normalfont 1)]
\item $f$ is a homomorphism. \item If $(P,Q)\in R_f$, then there
is $F\in X(H)$ such that $P\subseteq F$ and $f^{-1}(F) = Q$.
\end{enumerate}
\end{lem}

If $H$ is a Hilbert algebra and $a\in H$ we define
\begin{equation} \label{var}
\varphi(a): = \{P\in X(H): a\in P\}.
\end{equation}

Let $X$ and $Y$ be sets and let $R\subseteq X\times Y$ be a binary
relation. For every $x\in X$ we define $R(x): =\{y\in Y:(x,y)\in
R\}$.

Let $f:H\ra G$ an homomorphism in $\Hil$, $P\in X(G)$ and $a\in
H$. In \cite[Lemma 3.3]{CCM} it was proved that $f(a)\in P$ if and
only if for all $Q\in X(G)$, if $(P,Q) \in R_f$ then $a\in Q$. The
previous property can be written in the following way:
\begin{equation}\label{lem1}
f(a) \in P\;\text{if and only if}\; R_f(P) \subseteq \varphi(a).
\end{equation}

We now recall some definitions and results from \cite{CCM,CMs} and
fix some notation.

Let us consider a poset $\left\langle X,\leq\right\rangle$. For
each $Y \subseteq X$, \emph{the upset generated by $Y$} is defined
by $[Y) = \{x\in X:\;\text{there is}\;y\in Y\:\text{such
that}\;y\leq x\}$. The \emph{downset generated by $Y$} is dually
defined. If $Y = \{y\}$, then we will write $[y)$ and $(y]$
instead of $[\{y\})$ and $(\{y\}]$, respectively. We also define
$Y^{c}:= \{x\in X: x\notin Y\}$.

\begin{rem} \label{HH}
Let $\left\langle X,\leq\right\rangle$ be a poset. Write $X^{+}$
for the set of upsets of $\left\langle X,\leq\right\rangle$.
Define on $X^{+}$ the binary operation $\Ra$ by
\begin{equation}\label{imp}
U\Rightarrow V: =(U\cap V^{c}]^{c}.
\end{equation}
Then $X^{+}$ is a complete Heyting algebra.
\end{rem}

Let $(X,\tau)$ be a topological space. An arbitrary non-empty
subset $Y$ of $X$ is said to be \emph{irreducible} if for any
closed subsets $Z$ and $W$ such that $Y\subseteq Z \cup W$ we have
that $Y \subseteq Z$ or $Y \subseteq W$. We say that $(X,\tau)$ is
\emph{sober} if for every irreducible closed set $Y$ there exists
a unique $x\in X$ such that $Y = \clx$, where $\clx$ denotes the
closure of $\{x\}$. A subset of $X$ is \emph{saturated} if it is
an intersection of open sets. The \emph{saturation} of a subset
$Y$ of $X$ is defined as $\sat(Y):= \bigcap\{U\in \tau: Y\subseteq
U\}$. Recall that the \emph{specialization order} of $(X,\tau)$ is
defined by $x\preceq y$ if and only if $x \in \cly$. The relation
$\preceq$ is reflexive and transitive, i.e., a quasi-order. The
relation $\preceq$ is a partial order if $(X,\tau)$ is $T_0$. The
dual quasi-order order of $\preceq$ will be denoted by $\pd$.
Hence,
\[
x\pd y\;\text{if and only if}\; y\in \clx.
\]

\begin{rem}
Let $(X,\tau)$ be a topological space which is $T_0$, and consider
the order $\pd$. Let $x\in X$ and $Y\subseteq X$. Then $\clx =
[x)$ and $\sat(Y) = (Y]$, where $[x)$ is the upset generated by
$\{x\}$ with respect to the partial order $\pd$ and $(Y]$ is the
downset generated by $Y$ with respect to the partial order $\pd$.
\end{rem}

For the following definition see \cite{CMs}.

\begin{defn} \label{Hs}
A Hilbert space, or $H$-space for short, is a structure
$(X,\tau,\KA)$ where $(X,\tau)$ is a topological space, $\KA$ is a
family of subsets of $X$ and the following conditions are
satisfied:
\begin{enumerate} [\normalfont (H1)]
\item $\KA$ is a base of open and compact subsets for the topology
$\tau$ on $X$. \item For every $U, V \in \KA$, $\sat(U \cap V^{c})
\in \KA$. \item $(X, \tau)$ is sober.
\end{enumerate}
\end{defn}

In what follows, if $(X,\tau,\KA)$ is an $H$-space we simply write
$(X,\KA)$.

\begin{rem} \label{rimp}
\begin{enumerate} [\normalfont 1.]
\item A sober topological space is $T_0$. \item Viewing any
topological space as a poset, with the order $\pd$, condition
$\mathrm{(H2)}$ of Definition \ref{Hs} can be rewritten as: for
every $U,V\in \KA$, $(U\cap V^{c}] \in \KA$.
\end{enumerate}
\end{rem}

Let $X$ and $Y$ be sets and let $R\subseteq X\times Y$ be a binary
relation. If $U\subseteq Y$ then we define $R^{-1}(U): =\{x\in X:
R(x)\cap U \neq \emptyset\}$. Let $X,Y$ and $Z$ be sets,
$R\subseteq X\times Y$ and $S\subseteq Y\times Z$. Then the
relational product (or composition) of $R$ and $S$ is defined as
follows:
\begin{equation}\label{comp}
R\circ S: = \{(x,z):\;\text{there is}\; y\in Y\;\text{such that}\;
(x,y) \in R\;\text{and}\;(y,z)\in S\}.
\end{equation}

\begin{defn} \label{defHR}
Let $\X_1 = (X_1,\KA_1)$ and $\X_2 = (X_2,\KA_2)$ be two
$H$-spaces. Let us consider a relation $R\subseteq X_1 \times
X_2$. We say that $R$ is an $H$-relation from $\X_1$ into $\X_2$
if it satisfies the following properties:
\begin{enumerate} [\normalfont (HR1)]
\item $R^{-1}(U) \in \KA_1$, for every $U \in \KA_2$. \item $R(x)$
is a closed subset of $\X_2$, for all $x\in X_1$.
\end{enumerate}
We say that $R$ is an $H$-functional relation if it satisfies the
following additional condition:
\begin{enumerate}
\item[$\mathrm{(HF)}$] If $(x, y)\in R$ then there is $z\in X_1$
such that $z\in \clx$ and $R(z) = \cly$.
\end{enumerate}
\end{defn}

\begin{rem}
Condition $\mathrm{(HF)}$ from Definition \ref{defHR} can also be
given as follows: if $(x, y)\in R$ then there exists $z\in X_1$
such that $x\pd z$ and $R(z) = [y)$.
\end{rem}

If $H$ is a Hilbert algebra then $\X(H) = (X(H), \KA_H)$ is an
$H$-space, where $\KA_H =\{\varphi(a)^{c}: a\in H\}$. If $f$ is a
homomorphism of Hilbert algebras then $R_f$ is an $H$-functional
relation. Write $\HS$ for the category whose objects are Hilbert
spaces and whose morphisms are $H$-functional relations, where the
composition of two $H$-relations is defined as in (\ref{comp}).
Then the assignment $H\mapsto \X(H)$ can be extended to a functor
$\X:\Hil \ra \HS$.

Let $(X,\KA)$ be an $H$-space. Define $D(X) =\{U\subseteq
X:U^{c}\in \KA\}$. Then $D(X) \subseteq X^{+}$. It follows from
Definition \ref{Hs} and Remark \ref{rimp} that $D(X)$ is closed
under the operation $\Ra$ given in (\ref{imp}) of Remark \ref{HH}.
Since $X^{+}$ is a Heyting algebra then $\D(X) = (D(X),
\Rightarrow, X)$ is a Hilbert algebra. If $R$ is an $H$-functional
relation from $(X_1, \KA_1)$ into $(X_2, \KA_2)$, then the map
$h_R$ from from $\D(X_2)$ into $\D(X_1)$ given by $h_R(U) = \{x\in
X_1:R(x)\subseteq U\}$ is a homomorphism of Hilbert algebras.
Then the assignment $X\mapsto \D(X)$ can be extended to a functor
$\D:\HS\ra \Hil$.

If $H\in \Hil$, the map $\varphi: H \ra \D(\X(H))$ defined as in
$(\ref{var})$ is an isomorphism in $\Hil$. If $(X,\KA)$ is an
$H$-space, then the map $\epsilon_{X}: X\ra \X(\D(X))$ given by
$\epsilon_{X}(x) =\{U\in D(X): x\in U\}$ is an order-isomorphism
and a homeomorphism between the topological spaces $X$ and
$\X(\D(X))$(\cite[Theorem 2.2]{CCM}). If there is not ambiguity we
will write $\epsilon$ in place of $\epsilon_{X}$. Moreover, the
relation $\epsilon^{*}\subseteq X\times X(D(X))$ given by
$(x,P)\in \epsilon^{*}$ if and only if $\epsilon(x)\subseteq P$ is
an $H$-functional relation which is an isomorphism in $\HS$.

The following theorem can be found in \cite{CMs} (see also
\cite{CCM}).

\begin{thm}  \label{rh}
The contravariant functors $\X$ and $\D$ define a dual
equivalence between $\Hil$ and $\HS$ with natural
equivalences $\epsilon^{*}$ and $\varphi$.
\end{thm}

\section{An adjunction between $\Hil$ and $\IS$}\label{s1}

In this section we build up a functor from the algebraic category
of Hilbert algebras to the algebraic category of implicative
semilattices. This provides an explicit construction for the left
adjoint for the forgetful functor from the category of implicative
semilattices to that of Hilbert algebras. Finally we establish the
link between our result and the results studied in \cite{CJ2} (in
particular, with item (1) of \cite[Proposition 7.9]{CJ2}).

We start with some preliminary definitions and results.
\vspace{1pt}

Let $\langle H,\leq\rangle$ be a poset. If any two elements $a,b
\in H $ have a greatest lower bound $a\we b$, then the algebra
$(H,\we)$ is called \emph{meet semilattice}. The algebra $(H,\we)$
is said to be \emph{bounded} if it has a greatest element, which
will be denoted by $1$; in this case we write $(H,\we,1)$.
Throughout this paper we just write \emph{semilattice} in place of
meet semilattice.

\begin{defn}
An \emph{implicative semilattice} is an algebra $(H, \we, \ra)$ of
type $(2,2)$ such that $(H,\we)$ is a meet-semilattice and for
every $a,b,c\in H$, $a\we b \leq c$ if and only if $a\leq b\ra c$.
\end{defn}

In the literature implicative semilattices are known also as
Brouwerian semilattices. Every implicative semilattice has a
greatest element, denoted by $1$. In this paper we take this
element in the language of the algebras. It is part of the
folklore the fact that the variety of implicative semilattices is
the variety generated by the $\{1,\we,\ra\}$-reduct of Heyting
algebras. For more details about implicative semilattices see
\cite{Cu,N}.

We write $\IS$ for the category whose objects are implicative
semilattices and whose morphisms are functions $f:H\ra G$ such
that $f(1) = 1$ and $f(a\we b) = f(a) \we f(b)$ for every $a,b \in
H$.

\subsection{From $\IS$ to $\Hil$}

Let $f:H\ra G$ be a morphism in $\Hil$. It
follows from Theorem \ref{rh} that the following diagram commutes:

\[
 \xymatrix{
   H \ar[rr]^{\varphi} \ar[d]_{f} & & \D(\X(H))  \ar[d]^{\D(\X(f))}\\
   G \ar[rr]^{\varphi} & &  \D(\X(G)).
   }
\]

Let $g = \D(\X(f))$. The elements of $\D(\X(H))$ take the form
$\varphi(a)$ for $a\in H$. Thus, the commutativity of the previous
diagram is equivalent to the following equality, for every $a\in
H$:
\begin{equation} \label{cd}
\varphi(f(a)) = g(\varphi(a)).
\end{equation}
Also note that it follows from (\ref{lem1}) of Section \ref{BR}
that
\[
g(\varphi(a)) = \{P\in \X(G): R_{f}(P)\subseteq \varphi(a)\}.
\]

For every $H \in \Hil$ we have that $\varphi[H] = \D(\X(H))
\subseteq \X(H)^{+}$.

\begin{lem} \label{ext}
The homomorphism of Hilbert algebras $g$ can be extended to a
homomorphism of implicative semilattices $\hat{g}: \X(H)^{+} \ra
\X(G)^{+}$.
\end{lem}

\begin{proof}
Let $\hat{g}: \X(H)^{+} \ra \X(G)^{+}$ be given by
\[
\hat{g}(U) = \{P\in \X(G): R_{f}(P)\subseteq U\}.
\]
In order to show the good definition of $\hat{g}$, let $U \in
\X(H)^{+}$ and $P, Q \in \X(G)$ such that $P\subseteq Q$ and $P
\in \hat{g}(U)$, i.e., $R_{f}(P)\subseteq U$. Let $Z\in R_{f}(Q)$,
so $f^{-1}(Q)\subseteq Z$. Since $f^{-1}(P) \subseteq f^{-1}(Q)$
then $f^{-1}(P)\subseteq Z$, so $Z\in R_{f}(P)\subseteq U$. Thus,
$Z\in U$. Hence, $R_{f}(Q)\subseteq U$, i.e., $Q\in \hat{g}(U)$.
In consequence, $\hat{g}$ is a well defined map. It is immediate
that $\hat{g}(\X(H)) = \X(G)$ and $\hat{g}(U\cap V) = \hat{g}(U)
\cap \hat{g}(V)$ for every $U,V\in \X(H)^{+}$. In particular,
$\hat{g}(U\Rightarrow V) \subseteq \hat{g}(U) \Rightarrow
\hat{g}(V)$ for every $U,V\in \X(H)^{+}$.

Let $U,V\in \X(H)^{+}$. In order to prove that $\hat{g}(U)
\Rightarrow \hat{g}(V) \subseteq \hat{g}(U\Rightarrow V)$, suppose
that $P\notin \hat{g}(U \Rightarrow V)$, i.e., $R_{f}(P)\nsubseteq
U\Rightarrow V$. Then there exists $Q\in \X(H)$ such that
$f^{-1}(P)\subseteq Q$ and $Q\notin U\Rightarrow V$. Hence, there
exists $Z\in \X(H)$ such that $Q\subseteq Z$ and $Z\in U\cap V^c$.
Since $f^{-1}(P) \subseteq Q$ then it follows from Lemma
\ref{Th3.6} that there exists $W\in \X(G)$ such that $P\subseteq
W$ and $f^{-1}(W) = Z$. Thus, $R_{f}(W)\subseteq U$. In order to
show it, let $T\in R_{f}(W)$. Hence, $Z\subseteq T$. Since $Z\in
U$ and $U\in \X(H)^{+}$ then $T\in U$. So, we have proved that
$R_{f}(W)\subseteq U$. Besides $R_{f}(W)\nsubseteq V$ because
$Z\in R_{f}(W)$ and $Z\notin V$. Summarizing, $P\subseteq W$,
$R_{f}(W)\subseteq U$ and $R_{f}(W)\nsubseteq V$, so $P\notin
\hat{g}(U) \Rightarrow \hat{g}(V)$. Therefore, $\hat{g}(U
\Rightarrow V) = \hat{g}(U) \Rightarrow \hat{g}(V)$.
\end{proof}

\begin{rem}
In general, the map $\hat{g}$ of Lemma \ref{ext} is not
necessarily a Heyting homomorphism. In order to prove it, consider
the following two posets:

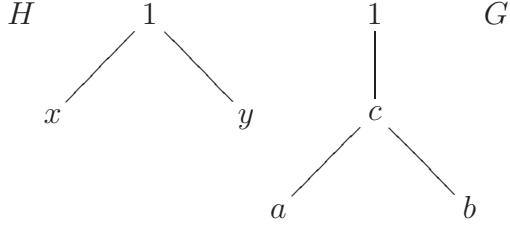
\begin{equation*}\label{d1}
\vspace{10pt}
H
\xymatrix{ & 1  &\\
\ar@{-}[ur] x  &  & \ar@{-}[ul] y \\
}
\\\\\\\\\\\\\\\\\\\\\\\\\\\\\\\\\\\\\\\\\\\\\\\\\\\\\\ \\\\\\\\\\\\\\\\\\\\\\\\\\\\\\\\\\\\\\\\\\\\\\\\\\\\\\ \\\\\\\\\\\\\\\\\\\\\\\\\\\\\\\\\\\\\\\\\\\\\\\\\\\\\\ \\\\\\\\\\\\\\\\\\\\\\\\\\\\\\\\\\\\\\\\\\\\\\\\\\\\\\
\\\\\\\\\\\\\\\\\\\\\\\\\\\\\\\\\\\\\\\\\\\\\\\\\\\\\\ \\\\\\\\\\\\\\\\\\\\\\\\\\\\\\\\\\\\\\\\\\\\\\\\\\\\\\ \\\\\\\\\\\\\\\\\\\\\\\\\\\\\\\\\\\\\\\\\\\\\\\\\\\\\\ \\\\\\\\\\\\\\\\\\\\\\\\\\\\\\\\\\\\\\\\\\\\\\\\\\\\\\
\xymatrix{ & 1  &\\
   & c  \ar@{-}[u]\\
\ar@{-}[ur] a  &  & \ar@{-}[ul] b \\
 }
 G
\end{equation*}

Endow these posets with the Hilbert algebra structures induced by the order; i.e.,
the implication is given by $x \ra y = 1$
if $x \leq y$ and $x \ra y = y$ if $x \nleq y$. Define $f:H \ra G$ by
$f(x) = a$, $f(y) = b$ and $f(1) = 1$. Straightforward
computations show that $f \in \Hil$ and that $\hat{f}$
does not preserve joins.
\end{rem}

It is worth mentioning that $\varphi[H] = \D(\X(H)) \subseteq \X(H)^{+}$, as defined before Lemma \ref{ext},
is the logic-based canonical extension of the Hilbert algebra $H$, as defined in \cite{GJP}. This fact follows from
\cite[Corollary 6.26]{Go18}.

\vskip3pt

If $S$ is a subset of an implicative semilattice, we write
$\langle S \rangle_{\IS}$ for the implicative semilattice
generated by $S$. Let $H\in \Hil$. Since $\varphi[H]$ is a subset
of the implicative semilattice $\X(H)^+$ then we define the
following implicative semilattice of $X(H)^+$:

\[
\HIS: = \langle \varphi[H] \rangle_{\IS}.
\]

Let $H\in \Hil$. It is interesting to note that the immersion of
$H$ into $\X(H)^+$ induced by $\varphi$ does not necessarily
preserve existing infima; however, it does preserve existing
suprema. See for instance \cite[Theorem 3.2]{CHJ}. Since $\X(H)^+$
is the logic-based canonical extension of H as defined in \cite{GJP}, then the
fact that $\varphi(a\vee b) = \varphi(a) \cup \varphi(b)$ whenever
$a \vee b$ exists can be also deduced from \cite[Proposition
6.10]{GJP2}. \vspace{3pt}

The following remark is a well known fact from universal algebra
\cite{BS}.

\begin{rem} \label{r1}
Let $A$ and $B$ be algebras of the same type and $X\subseteq A$
with $X\neq \emptyset$. Let $f:A\ra B$ be a homomorphism. Write
$\Sg^{A}(X)$ for the subalgebra of $A$ generated by $X$ and
$\Sg^{B}(f(X))$ for the subalgebra of $B$ generated by $f(X)$. We
have that $f(\Sg^{A}(X))= \Sg^{B}(f(X))$.
\end{rem}

\begin{lem} \label{morf}
The homomorphism of implicative semilattices $\hat{g}$ defined in
Lemma \ref{ext} satisfies $\hat{g}(\HIS) \subseteq \GIS$.
\end{lem}

\begin{proof}
It follows from Lemma \ref{ext}, Remark \ref{r1} and the equality
$g(\varphi(a)) = \varphi(f(a))$ given in (\ref{cd}).
\end{proof}

Let $f:H \ra G$ be a morphism in $\Hil$. It follows from lemmas
\ref{ext} and \ref{morf} that the map $\fIS: \HIS \ra \GIS$ given
by
\[
\fIS(U) = \{P\in \X(G): R_{f}(P)\subseteq U\}
\]
is a morphism in $\IS$. Let $\mathsf{Id}$ be an identity morphism
in $\Hil$. It is immediate that $\mathsf{Id}^{\mathsf{IS}}$ is an
identity in $\IS$. 
Let $f:H\ra G$ and $g:G \ra K$ be morphisms in $\Hil$. It follows
from \cite[Theorem 3.3]{CCM} that $R_{g\circ f} = R_{g} \circ
R_{f}$. Hence, straightforward computations based in the above
mentioned equality shows that
\begin{equation*}
\label{isafunctor}
(g \circ f)^{\mathsf{IS}} = \gIS \circ \fIS.
\end{equation*}


Therefore we obtain the following proposition.

\begin{prop} \label{propfun}
The assignments $H\mapsto \HIS$ and $f\mapsto \fIS$ define a
functor $\f:\Hil \ra \IS$.
\end{prop}

In what follows, we write $\UIS$ for the forgetful functor from
$\IS$ to $\Hil$.

\subsection{Adjunction}

Now we prove that the functor $\f: \Hil \ra \IS$ is left adjoint
of $\UIS$.

Recall that if $H\in \IS$, a subset $F\subseteq H$ is said to be a
\emph{filter} if it satisfies the following conditions:
\begin{enumerate}[\normalfont 1)]
\item $1 \in F$, \item $a\we b \in F$ whenever $a,b \in F$, \item
$F$ is an upset.
\end{enumerate}
We can also define the concept of implicative (and irreducible)
filter for the case of implicative semilattices. It is part of the
folklore that if $H\in \IS$ then the set of implicative filters of
$H$ is equal to the set of filters of $H$. If $H\in \IS$ we also
write $X(H)$ for the set of irreducible filters of $H$.

\begin{rem}
Let $H\in \IS$. We write $\HIS$ in place of
$(\UIS(H))^{\mathsf{IS}}$. For every $a\in H$ we also write
$\varphi(a)$ for the set $\{P\in X(H):a\in P\}$.
\end{rem}

Let $H\in \Hil$. Consider the injective morphism of Hilbert
algebras $\psi:H\ra \UIS(\HIS)$ given by $\psi(a) = \varphi(a)$.

\begin{prop} \label{pu}
Let $G\in \IS$ and $f:H\ra \UIS(G)\in \Hil$. Then, there exists a
unique morphism $h:\HIS \ra G$ such that $f = \U(h) \circ \psi$.
\end{prop}

\begin{proof}
The map $\fIS:\ \HIS \ra \GIS$ is a morphism in $\IS$. Since $G\in
\IS$ then for every $a,b\in G$ we have that $\varphi(a\we b) =
\varphi(a) \cap \varphi(b)$, so a reflection's moment shows that
the map $\varphi:G\ra \GIS$ is an isomorphism in $\IS$. Hence, the
map $h:\HIS \ra G$ given by $h = \varphi^{-1} \circ \fIS$ is also
a morphism in $\IS$. Finally, it follows from (\ref{cd}) that $f =
\U(h) \circ \psi$.
\end{proof}


Let $\IHil$ be the identity functor in $\Hil$. It follows from
(\ref{cd}) that $\Psi: \IHil \ra \UIS \circ \f$ is a natural
transformation. Here, the family of morphism associated to the
natural transformation is given by the morphisms $\psi$.

In other words, to say that $\Psi: \IHil \ra \UIS \circ \f$ is a
natural transformation is equivalent to say that if $f:H\ra G$ is
a morphism in $\Hil$ then the following diagram commutes:

\[
 \xymatrix{
   H \ar[rr]^{f} \ar[d]_{\psi} & & G  \ar[d]^{\psi}\\
   \UIS(\HIS) \ar[rr]^{\UIS(\fIS)} & & \UIS(\GIS). \
   }
\]

Therefore we get the following result.

\begin{thm} \label{pteo}
The functor $\f: \Hil \ra \IS$ is left adjoint to $\UIS$.
\end{thm}

\subsection{Connections with the literature}

In what follows we connect our results with those of \cite{CJ2}.
We also make a brief remark about $\varphi_H : H \ra \X(H)^{+}$,
viewed as the logic-based canonical extension of the Hilbert algebra $H$,
as presented in  \cite{GJP}.

Fix $H\in \Hil$. A pair $(G,e)$ where $G$ is an implicative
semilattice and $e$ is an injective morphism from $H$ to $\UIS(G)$
is said to be an \emph{implicative semilattice envelope of $H$} if
for every $y\in G$ there exists a finite subset $X\subseteq H$
such that $y = \bigwedge e(X)$. Define the following set:
\[
\FC(H) = \{U : U = \varphi(a_1)\cap \cdots \cap
\varphi(a_n)\;\text{for some}\; a_1,\ldots,a_n \in H\}.
\]
Then, $\varphi[H] \subseteq \FC(H) \subseteq \X(H)^{+}$. Moreover,
$\FC(H) \in \IS$ by considering the implication $\Ra$ given by
(\ref{imp}). The pair $(\FC(H),\eta)$ is an implicative
semilattice envelope of $H$, where $\eta: H \ra \UIS(\FC(H))$ is
given by $\eta(a) = \varphi(a)$ \cite[Lemma 6.4]{CJ2}. It follows
from results of \cite{CJ2} that $(\FC(H),\eta)$ is a solution
$(G,e)$ of the following universal problem: For every $G' \in \IS$
and every $e' : H \ra \UIS(G') \in \Hil$, there is a unique $g: G
\ra G' \in \IS$ such that $e' = \UIS(g) \circ e$.

It follows from \cite[Proposition 6.9]{CJ2} that if $h : H \ra K
\in \Hil$ then there is a unique $\overline{h}: \Ser(H) \ra
\Ser(K)$ such that $\eta \circ h = \overline{h} \circ \eta$.

The following theorem summarize some properties from \cite{CJ2},
which it was proved in an alternative way in the present paper
(Theorem \ref{pteo}).

\begin{thm} \label{pcj}
There exists a functor $\Ser:\Hil\ra \IS$ that maps every $H \in
\Hil$ to $\Ser(H) \in \IS$, and every $h : H\ra G \in \Hil$ to
$\overline{h}:\Ser(H) \ra \Ser(G)\in \IS$. The functor $\Ser$ is
left adjoint to $\UIS$.
\end{thm}

Let $H\in \Hil$. Since $\FC(H) \in \IS$, it follows from the
definition of $\HIS$ that $\HIS$ is in fact equal to $\FC(H)$,
i.e.,
\begin{equation}\label{ed}
\HIS = \FC(H).
\end{equation}

In \cite{GJP}, the \emph{logic-based canonical extension} of an algebra in certain
classes of algebras of interest for abstract algebraic logic is presented. In particular,
it is proved there that when this notion of canonicity is considered, the variety of
Hilbert algebras is canonical. It can be seen that the logic-based canonical extension of
a Hilbert algebra, as performed in \cite[Section 5.1]{GJP} is indeed given by the embedding
$\varphi_H : H \ra \X(H)^{+}$.

\section{Relation between $\Hils$ $(\Hilsz$) and $\GHey$ ($\Hey$)}
\label{s2}

In this section we provided a construction for the left adjoint of
the forgetful functor from the algebraic category of generalized
Heyting algebras (Heyting algebras) to the algebraic category of
Hilbert algebras with supremum (Hilbert algebras with supremum and
a minimum).

\begin{defn}
An algebra $(H,\vee,\ra,1)$ of type $(2,2,0)$ is a \emph{Hilbert
algebra with supremum} if the following conditions are satisfied:
\begin{enumerate}
\item[\rm{1.}] $(H,\ra,1)$ is a Hilbert algebra. \item[\rm{2.}]
$(H,\vee,1)$ is a join semilattice with greatest element $1$.
\item[\rm{3.}] For every $a,b\in H$, $a\ra b = 1$ if and only if
$a \vee b = b$.
\end{enumerate}
\end{defn}

We denote by $\Hils$ to the category whose objects are Hilbert
algebras with supremum and whose morphisms are the homomorphisms
$f:H\ra G$ in $\Hil$ such that $f(a\vee b) = f(a)\vee f(b)$ for
every $a,b\in H$. For more details about Hilbert algebras with
supremum see \cite{CMs}.

\begin{defn}
A \emph{generalized Heyting algebra} ($gH$-algebra for short) is a
lattice such that for every $a,b\in H$ there exists the maximum of
the set $\{c\in H: a\we c\leq b\}$, denoted by $a\ra b$.
\end{defn}

It is known that $gH$-algebras have a largest element, which will
be denoted by $1$. We consider $gH$-algebras as algebras
$(H,\we,\vee,\ra,1)$ of type $(2,2,2,0)$, and Heyting algebras as
algebras $(H,\we,\vee,\ra,0,1)$ of type $(2,2,2,0,0)$. We write
$\GHey$ for the category of $gH$-algebras and $\Hey$ for the
category of Heyting algebras. For more about these classes of
algebras see \cite{BD,Mon}.

Let $H\in \Hils$ and $F$ an implicative filter of $H$. We say that
$F$ is \emph{prime} if it is proper and for any $a,b \in H$ such
that $a\vee b \in F$ we have that $a\in F$ or $b\in F$. It is part
of the folklore that the set of prime filters of $H$ is equal to
the set of irreducible implicative filters of $H$.

\begin{lem} \label{ls1}
Let $f:H\ra G$ be a morphism in $\Hils$. Then for every $P\in
\X(G)$ it holds that $f^{-1}(P)\in \X(H)$ or $f^{-1}(P) = H$.
\end{lem}

\begin{proof}
It follows from \cite[Lemma 5.10]{CMs}.
\end{proof}

In \cite[Lemma 5.10]{CMs} it was also proved that if $f:H\ra G$ be
a morphism in $\Hils$ then for every $P\in \X(G)$ such that
$R_{f}(P) \neq \emptyset$ it holds that if $U, V$ are closed sets
of $\X(H)$ such that $R_{f}(P) \subseteq U \cup V$ then $R_{f}(P)
\subseteq U$ or $R_{f}(P)\subseteq V$. In the following lemma we
generalize the above mentioned property.

\begin{lem} \label{ls2}
Let $f:H\ra G$ be a morphism in $\Hils$, $P\in \X(G)$ and $U,V \in
\X(H)^{+}$. If $R_{f}(P)\subseteq U\cup V$ then $R_{f}(P)
\subseteq U$ or $R_{f}(P) \subseteq V$.
\end{lem}

\begin{proof}
Assume that $R_{f}(P)\subseteq U\cup V$. It follows from Lemma
\ref{ls1} that $f^{-1}(P) \in \X(H)$ or $f^{-1}(P) = H$. Suppose
that $f^{-1}(P) = H$, i.e., $R_{f}(P) = \emptyset$. Then
$R_{f}(P)\subseteq U$ or $R_{f}(P)\subseteq V$. Hence, we can
assume that $f^{-1}(P) \in \X(H)$, i.e., $f^{-1}(P) \in R_{f}(P)$.
Since $R_{f}(P)\subseteq U\cup V$ then $f^{-1}(P)\in U$ or
$f^{-1}(P) \in V$. Consider that $f^{-1}(P) \in U$ and let $Q\in
R_{f}(P)$. Thus, $f^{-1}(P) \subseteq Q$. Since $f^{-1}(P) \in U$
and $U\in \X(H)^{+}$ then $Q\in U$. Hence, $R_{f}(P) \subseteq U$.
The same argument proves that if $f^{-1}(P) \in V$ then $R_{f}(P)
\subseteq V$.
\end{proof}

If $S$ is a subset of a $gH$-algebra, we write $\langle S
\rangle_{\GHey}$ for the $gH$-algebra generated by $S$. Let $H\in
\Hil$. Since $\varphi[H]$ is a subset of the $gH$-algebra
$\X(H)^+$ then we define the following $gH$-algebra of $X(H)^+$:
\[
H^{\GHey}: = \langle \varphi[H] \rangle_{\GHey}.
\]

Let $f:H\ra G$ be a morphism in $\Hils$ and  $\hat{g}: \X(H)^{+}
\ra \X(G)^{+}$ the function considered in Lemma \ref{ext}. It
follows from Lemma \ref{ls2} that if $U,V \in \X(H)^{+}$ then
\[
\hat{g}(U\cup V) = \hat{g}(U) \cup \hat{g}(V).
\]
Hence, by Lemma \ref{ext} we have that $\hat{g}: \X(H)^{+} \ra
\X(G)^{+}$ is a morphism in $\GHey$. By the same argument of Lemma
\ref{morf} we have that $\hat{g}(H^{\GHey})\subseteq G^{\GHey}$,
so $\hat{g}$ can be restricted to a morphism $f^{\GHey}: H^{\GHey}
\ra G^{\GHey}$ in $\GHey$.

The following result is similar to that given in Proposition
\ref{propfun}.

\begin{prop}
The assignments $H\mapsto H^{\GHey}$ and $f\mapsto f^{\GHey}$
define a functor $\fs:\Hils \ra \GHey$.
\end{prop}

Write also $\UGHey$ for the forgetful functor from $\GHey$ to
$\Hils$. As in the case of Theorem \ref{pteo}, we can show the
following result.

\begin{thm}
The functor $\fs: \Hils \ra \GHey$ is left adjoint to $\UGHey$.
\end{thm}

In the following proposition we give an easy description for
$H^{\GHey}$ when $H$ is a finite algebra.

\begin{prop}
Let $H\in \Hils$ be a finite algebra. Then $H^{\GHey} =
\X(H)^{+}$.
\end{prop}

\begin{proof}
By definition we have that $H^{\GHey}\subseteq \X(H)^{+}$. In
order to prove the converse inclusion, let $U\in \X(H)^+$. Since
$H$ is finite then $\emptyset \in H^{\GHey}$, so we can assume
that $U\neq \emptyset$. Since $U$ is finite there exist
$P_1,\ldots,P_n\in \X(H)$ such that $U = \{P_1,\ldots,P_n\}$.
Hence, $U = \bigcup_{i=1}^{n} [P_i)$. Note that if $P\in \X(H)$
then $[P) = \bigcap_{a\in P} \varphi(a)$ (it is a finite union)
and $\varphi(1) = \X(H)$. Thus, $U = \bigcap_{i=1}^{m} U_i$, where
for every $i=1,\ldots,m$ the set $U_i$ takes the form
$\varphi(a_{1i})\cap \cdots \cap \varphi(a_{mi})$. Then $U\in
H^{\GHey}$. Therefore, $H^{\GHey} = \X(H)^{+}$.
\end{proof}

For $H\in \Hils$ we can give the following description of
$H^{\GHey}$, where we see $H^{\GHey}$ as an implicative
semilattice.

\begin{prop} \label{aG}
Let $H\in \Hils$. Then $\FC(H) = H^{\GHey}$.
\end{prop}

\begin{proof}
By (\ref{ed}) of the end of Section \ref{s1} we only need to prove
that if $U, V \in \FC(H)$ then $U\cup V \in \FC(H)$. Let $U, V \in
\FC(H)$. Then there exist $a_1,\ldots,a_n \in H$ and
$b_1,\ldots,b_m\in H$ such that $U = \varphi(a_1) \cap \cdots \cap
\varphi(a_n)$ and $V = \varphi(b_1) \cap \cdots \cap
\varphi(b_m)$. Since $\varphi(1) = \X(H)$ we can assume that $n =
m$. Hence, $U \cup V = \bigcap_{i,j=1}^{n} (\varphi(a_i)\cup
\varphi(b_j)) = \bigcap_{i,j=1}^{n} \varphi(a_i \vee b_j)$.
Therefore, $U\cup V \in \FC(H)$.
\end{proof}

Let $\Hilsz$ be the category whose objects are algebras
$(H,\vee,\ra,0,1)$ of type $(2,2,0,0)$ such that $(H,\vee,\ra,1)
\in \Hils$ and $0$ satisfies that $0\leq x$ for every $x\in H$.
The morphisms are the homomorphisms $f$ of $\Hils$ such that $f(0)
= 0$. Also write $\UHey$ for the forgetful functor from $\Hey$ to
$\Hils$. If $H\in \Hey$ we define $H^{\Hey}$ as the Heyting
subalgebra of $\X(H)^+$ generated by $\varphi[H]$, and if $f:H\ra
G$ is a morphism in $\Hey$ we can define a morphism $f^{\Hey}:
H^{\Hey} \ra G^{\Hey}$ in $\Hey$ similarly to the case of
$gH$-algebras. We also write $\UHey$ for the forgetful functor
from $\Hey$ to $\Hilsz$

\begin{cor}
The functor $\fsz: \Hilsz \ra \Hey$ is left adjoint to $\UHey$.
\end{cor}

\section{Final remarks} \label{fr}

In this final section we define a functor from $\Hil$ to $\GHey$.
As usual, we start with some definitions and preliminary results.

The proof of the following lemma is similar to the proof of
\cite[Theorem 3.3]{Cel}. 

\begin{lem} \label{fh1}
Let $f:H\ra G\in \Hil$, $I\in \Fil(G)$ and $J\in \X(H)$ be such
that $f^{-1}(I)\subseteq J$. Then there exists $K\in \X(H)$ such
that $I\subseteq K$ and $f^{-1}(K) = J$.
\end{lem}

\begin{proof}
Since $J\in \X(H)$ we have that
\[
(f(J^{c})] = \{b\in G:b\leq f(a)\;\text{for some}\;a\in J^c\}
\]
is an order ideal of $G$ (see \cite[Theorem 2.3]{Cel}). Let us see
that
\[
(f(J^{c})] \cap F(I\cup f(J)) = \emptyset.
\]
Suppose that $(f(J^{c})] \cap F(I\cup f(J)) \neq \emptyset$. Then,
there are $x \in H$, $j \notin J$, $j_1, \ldots, j_m \in J$ and
$i_1, \ldots, i_n \in I$ such that $x \leq f(j)$ and
\[
i_1 \ra (i_2 \ra  \cdots (i_n \ra ( f(j_1) \ra (f(j_2) \ra  \cdots
(f(j_m) \ra x) \ldots ) = 1.
\]
Note that elements can be always ordered in this way, since in any
Hilbert algebra the identity $a \ra (b \ra c) = b \ra (a \ra c)$
holds. Since $1, i_1, \ldots, i_n$ are in $I$, $f(j_1) \ra (f(j_2)
\ra  \cdots (f(j_m) \ra x) \ldots ) \in I$. Since $x \leq f(j)$
then
\begin{eqnarray*}
f(j_1) \ra (f(j_2) \ra  \cdots (f(j_m) \ra x) \ldots ) \leq &  \\
f(j_1) \ra (f(j_2) \ra  \cdots (f(j_m) \ra f(j)) \ldots ) = &
f(j_1 \ra (j_2 \ra  \cdots (j_m \ra j) \ldots )).
\end{eqnarray*}
Since $I$ is an upset then $f(j_1 \ra (j_2 \ra  \cdots (j_m \ra j)
\ldots )) \in I$, what implies that $j_1 \ra (j_2 \ra  \cdots (j_m
\ra j) \ldots ) \in f^{-1}(I) \subseteq J$. Hence, $j \in J$,
which is a contradiction. In consequence, $(f(J^{c})] \cap F(I\cup
f(J)) = \emptyset$.

By Lemma \ref{tfp}, there exists $K \in \X(G)$ such that $F(I\cup
f(J))\subseteq K$ and $(f(J^{c})] \cap K = \emptyset$. Thus,
$I\subseteq K$ and $f^{-1}(K) = J$.
\end{proof}

Let $H\in \Hil$. We define the following set:
\begin{equation}\label{specext}
\Xs(H) = \{F\in \Fil(H): F = \bigcap \textit{X}_0,\;\text{for some
finite}\;\textit{X}_0\subseteq \X(H)\}.
\end{equation}

It is known that if $F$ is a proper implicative filter of $H$ then
$F$ is the intersection of all irreducible filters of $H$ such
that contain $F$ (it is an immediate consequence of Corollary
\ref{tfpc1}). In particular, note that if $H$ is finite then $F
\in \Xs(H)$ if and only if $F$ is a proper implicative filter of
$H$.

\begin{cor} \label{fh2}
Let $f:H\ra G\in \Hil$, $I \in \Xs(H)$ and $J\in \Xs(G)$ be such
that $f^{-1}(I) \subseteq J$. Then there exists $K\in \Xs(H)$ such
that $I \subseteq K$ and $f^{-1}(K) = J$.
\end{cor}

\begin{proof}
Let $I \in \Xs(H)$ and $J\in \Xs(G)$ be such that $f^{-1}(I)
\subseteq J$. Then there exist $Q_1,\ldots,Q_n \in \X(G)$ such
that $J = Q_1\cap\cdots \cap Q_n$, so $f^{-1}(I)\subseteq Q_i$ for
every $i=1,\ldots,n$. By Lemma \ref{fh1} we have that there exist
$K_1,\ldots,K_n\in \X(H)$ such that $I\subseteq K_i$ and
$f^{-1}(K_i) = Q_i$ for every $i=1,\ldots,n$. Let $K = \K_{1} \cap
\cdots \cap K_n$. Thus, $K\in \Xs(H)$, $I\subseteq K$ and
$f^{-1}(K) = J$.
\end{proof}

Let $H\in \Hil$. We define the function $\Phi:H \ra (\Xs(H))^{+}$
by
\[
\Phi(a) = \{F\in \Xs(H): a\in F\}.
\]

\begin{lem}\label{fh3}
Let $H\in \Hil$. The function $\Phi$ defined above is an injective
morphism in $\Hil$. Moreover, if $H^{\dag}$ is the $gH$-algebra of
$(\Xs(H))^{+}$ generated by $\Phi(H)$ then $\Phi$ can be also
considered as a map from $H$ to $H^{\dag}$.
\end{lem}

\begin{proof}
It is immediate that if $a\in H$ then $\Phi(a)$ is an upset of
$\Xs(H)$ and that $\Phi(1) = \Xs(H)$. The equality $\Phi(a\ra b) =
\Phi(a) \Rightarrow \Phi(b)$ follows from Corollary \ref{tfpc1}.
Finally, Corollary \ref{tfpc2} implies the injectivity of $\Phi$.
\end{proof}

Let $f:H\ra G\in \Hil$. We define $\Rfs \subseteq \Xs(G) \times
\Xs(H)$ by
\[
(I,J) \in \Rfs\; \text{if and only if}\; f^{-1}(I)\subseteq J.
\]

\begin{lem}\label{fh4}
Let $f:H\ra G\in \Hil$. Then the following holds:
\begin{enumerate} [\normalfont a)]
\item For every $a\in H$, $\Phi(f(a)) = \{F\in \Xs(G):
\Rfs(F)\subseteq \Phi(a)\}$. \item The function $g: (\Xs(H))^{+}
\ra (\Xs(G))^{+}$ given by
\[
g(U) = \{F\in \Xs(G): \Rfs(F) \subseteq U\}
\]
is a morphism in $\GHey$. \item For every $a\in H$, $g(\Phi(a)) =
\Phi(f(a))$. In particular, $f(H^{\dag})\subseteq G^{\dag}$ and
the function $f^{\dag}: H^{\dag} \ra G^{\dag}$ given by
$f^{\dag}(U) = g(U)$ is a morphism in $\GHey$.
\end{enumerate}
\end{lem}

\begin{proof}
First we will prove $a)$. Let $a\in H$. It is immediate that
$\Phi(f(a)) \subseteq \{F\in \Xs(G): \Rfs(F)\subseteq \Phi(a)\}$.
Conversely, let $F\in \Xs(G)$ such that $\Rfs(F) \subseteq
\Phi(a)$. Suppose that $F\notin \Phi(f(a))$, i.e., $a \notin
f^{-1}(F)$. Since $f^{-1}(F)$ is an implicative filter of $H$ then
it follows from Corollary \ref{tfpc3} that there exists $P\in
\X(H)$ such that $f^{-1}(F)\subseteq P$ and $a\notin P$. Then,
$P\in \Rfs(F) \subseteq \Phi(a)$. Hence, we deduce that $a\in P$,
which is a contradiction. Thus, we have proved that $\{F\in
\Xs(G): \Rfs(F)\subseteq \Phi(a)\} \subseteq \Phi(f(a))$.

Now we will prove $b)$. Let $U,V\in (\Xs(H))^{+}$. It is immediate
that $g(U\cap V) = g(U)\cap g(V)$ and $g(\Xs(H)) = \Xs(G)$. The
equality $g(U\Rightarrow V) = g(U) \Rightarrow g(V)$ can be proved
as Lemma \ref{ext} but using Corollary \ref{fh2}.

Finally we will prove that $g(U\cup V) = g(U) \cup g(V)$ for every
$U,V\in (\Xs(H))^{+}$. It is enough to prove that if $F\in \Xs(G)$
and $\Rfs(F)\subseteq U\cup V$ then $\Rfs(F) \subseteq U$ or
$\Rfs(F) \subseteq V$. Let $\Rfs(F)\subseteq U\cup V$. Suppose
that $\Rfs(F)\nsubseteq U$ and $\Rfs(F)\nsubseteq V$. Then there
exist $J,K \in \Xs(H)$ such that $f^{-1}(F) \subseteq J$,
$f^{-1}(F)\subseteq K$, $J\notin U$ and $K\notin V$. In
particular, $f^{-1}(F) \subseteq J \cap K$. Since $J\cap K\in
\Xs(H)$ then $J\cap K \in \Rfs(F) \subseteq U\cup V$, so $J\cap K
\in U$ or $J\cap K \in V$. Since $J\cap K \subseteq J$, $J\cap K
\subseteq K$ and $U,V \in (\Xs(H))^{+}$ then $J\in U$ or $K\in V$,
which is a contradiction.

The item $c)$ follow from items $a)$ and $b)$.
\end{proof}

The following result follows from Lemma \ref{fh3}.

\begin{prop}
The assignments $H\mapsto H^{\dag}$ and $f\mapsto f^{\dag}$ define
a functor $\fHil:\Hil \ra \GHey$.
\end{prop}

A straightforward computation shows that if $H$ is a finite
Hilbert algebra then
\[
H^{\dag} = (\Xs(H))^{+}.
\]

We conclude this paper stating the following open problem. Is it
possible to adapt the constructions of this paper in order to get
an explicit description of the (generalized) Heyting algebra
freely generated by a Hilbert algebra?

\subsection*{Compliance with Ethical Standards}

This work was supported by CONICET-Argentina [PIP
112-201501-00412]. The authors thank the anonymous referee for the
useful comments on the manuscript. In particular, for making us
notice the connection of our construction with the logic-based
canonical extension of Hilbert algebras. The second author would
also like to thank Ram\'on Jansana for useful discussions
concerning the results of this work.

Both authors declare that they have no conflict of interest. This
article does not contain any studies with animals or humans
performed by any of the authors.

\end{document}